\documentclass{amsart}

\usepackage{amsmath,amssymb,enumerate,epic,graphicx,mathrsfs}

\newtheorem{teor}{Theorem}
\newtheorem{lemma}{Lemma}
\newtheorem{cor}{Corollary}
\newtheorem{prop}{Proposition}

\DeclareMathOperator{\sol}{sol}
\DeclareMathOperator{\cp}{cp}
\DeclareMathOperator{\Alt}{Alt}
\DeclareMathOperator{\Sym}{Sym}
\DeclareMathOperator{\Aut}{Aut}
\DeclareMathOperator{\PSL}{PSL}
\DeclareMathOperator{\PGL}{PGL}

\begin{document}

\title{On the number of cyclic subgroups of a finite group}
\date{}
\author{Martino Garonzi}
\address[Martino Garonzi]{Departamento de Matem\'atica, Universidade de Bras\'ilia, Campus Universit\'ario Darcy Ribeiro, Bras\'ilia-DF, 70910-900, Brazil}
\email{mgaronzi@gmail.com}
\author{Igor Lima}
\address[Igor Lima]{Universidade Federal de Goi\'{a}s, IMTec - Regional Catal\~{a}o, Av. Dr. Lamartine P. de Avelar, 1120 Setor Universit\'ario Catal\~{a}o-GO, 75704-020, Brazil}
\email{igor.matematico@gmail.com}

\begin{abstract}
Let $G$ be a finite group and let $c(G)$ be the number of cyclic subgroups of $G$. We study the function $\alpha(G) = c(G)/|G|$. We explore its basic properties and we point out a connection with the probability of commutation. For many families $\mathscr{F}$ of groups we characterize the groups $G \in \mathscr{F}$ for which $\alpha(G)$ is maximal and we classify the groups $G$ for which $\alpha(G) > 3/4$. We also study the number of cyclic subgroups of a direct power of a given group deducing an asymptotic result and we characterize the equality $\alpha(G) = \alpha(G/N)$ when $G/N$ is a symmetric group.
\end{abstract}

\maketitle

\section{Introduction}

In this paper all the groups we consider are finite. Let $c(G)$ be the number of cyclic subgroups of a group $G$ and $\alpha(G) := c(G)/|G|$. It is clear that $0 < \alpha(G) \leq 1$. Observe that every cyclic subgroup $\langle x \rangle$ of $G$ has $\varphi(o(x))$ generators, where $\varphi$ is Euler's totient function and $o(x)$ denotes the order of the element $x$, hence $$c(G) = \sum_{x \in G} \frac{1}{\varphi(o(x))}.$$ On a computational level this formula is probably best for computing $c(G)$ for arbitrary $G$, and it is what we employed to work out small groups using \cite{gap}. For $d$ a divisor of $|G|=n$ let $B_G(d)$ be the number of elements $x \in G$ such that $x^d=1$. Denote by $\mu$ the standard M\"obius function. In \cite{garpat} another formula is given (from Lemma 7 (2) choosing $A(d)$ the number of elements of $G$ of order $d$, $B(d)=B_G(d)$ and $(r,s)=(1,0)$), which is an easy application of the M\"obius inversion formula, the following:
\begin{equation} \label{moe}
c(G) = \sum_{d|n} \left( \sum_{i|n/d} \frac{\mu(i)}{\varphi(di)} \right) B_G(d).
\end{equation}
Using this formula in \cite[Corollary 13]{garpat} it was shown that ciclicity can be detected by the number of cyclic subgroups, more precisely that if $|G|=n$ then $c(G) \geq c(C_n)$ with equality if and only if $G \cong C_n$.

\ 

There is a connection between $\alpha(G)$ and the so-called ``commuting probability'' of $G$, denoted by $\cp(G)$, that is the probability that two random elements of $G$ commute (studied extensively in \cite{gurob}, which we crucially employ in our study). More specifically we prove that if $\alpha(G) \geq 1/2$ then $\cp(G) \geq (2 \alpha(G)-1)^2$. This implies that a group with many cyclic subgroups has big solvable radical and, if it is already solvable, it has big Fitting subgroup (see Section \ref{cpsection} for the details).

\ 

It is not hard to show that $\alpha(G) \leq \alpha(G/N)$ whenever $N$ is a normal subgroup of $G$, and if equality holds then $N$ must be an elementary abelian $2$-group. It is an interesting question to ask what can we say about $G$ if $\alpha(G)=\alpha(G/N)$ given some information on $G/N$. In this paper we characterize equality $\alpha(G) = \alpha(G/N)$ when $G/N$ is a symmetric group (Theorem \ref{sym}).

\ 

Given a family $\mathscr{F}$ of groups define $$\alpha_{\mathscr{F}} := \max \{\alpha(G)\ :\ G \in \mathscr{F}\},\hspace{1cm} m\mathscr{F} := \{G \in \mathscr{F}\ :\ \alpha(G) = \alpha_{\mathscr{F}}\}.$$We are interested in computing $\alpha_{\mathscr{F}}$ and $m\mathscr{F}$ for various families $\mathscr{F}$. In this paper we prove the following results.

\ 

\begin{enumerate}
\item If $\mathscr{F}$ is the family of all finite groups then $\alpha_{\mathscr{F}}=1$ and $m\mathscr{F}$ is the family of elementary abelian $2$-groups (by \ref{2a-1}).

\ 

\item If $\mathscr{F}$ is the family of non-abelian groups then $$\alpha_{\mathscr{F}} = 7/8 = \alpha(D_8)$$ and $m \mathscr{F}$ is the family of groups of the form ${C_2}^n \times D_8$ for some $n \geq 0$ (by Corollary \ref{nil}).

\ 

\item If $\mathscr{F}$ is the family of non-nilpotent groups then $$\alpha_{\mathscr{F}} = 5/6 = \alpha(S_3)$$ and $m \mathscr{F}$ is the family of groups of the form ${C_2}^n \times S_3$ for some $n \geq 0$ (by Corollary \ref{nil}).

\ 

\item If $\mathscr{F}$ is the family of non-solvable groups then $$\alpha_{\mathscr{F}} = 67/120 = \alpha(S_5)$$ and $m \mathscr{F}$ is the family of groups of the form ${C_2}^n \times S_5$ for some $n \geq 0$ (Theorem \ref{sol}).

\ 

\item If $\mathscr{F}$ is the family of non-supersolvable groups then $$\alpha_{\mathscr{F}} = 17/24 = \alpha(S_4)$$ and $m \mathscr{F}$ is the family of groups of the form ${C_2}^n \times S_4$ for some $n \geq 0$ (Theorem \ref{nonsupersol}).

\ 

\item If $p$ is an odd prime and $\mathscr{F}$ is the family of non-trivial groups of order divisible only by primes at least $p$ then $$\alpha_{\mathscr{F}} = 2/p = \alpha(C_p)$$ and $m \mathscr{F} = \{C_p\}$ (Proposition \ref{odd}).

\ 

\item If $p$ is an odd prime and $\mathscr{F}$ is the family of groups $G$ with $C_p$ as an epimorphic image then $$\alpha_{\mathscr{F}}=2/p=\alpha(C_p)$$ and $m \mathscr{F}$ is the family of groups which are the direct product of an elementary abelian $2$-group and a Frobenius group with $2$-elementary abelian Frobenius kernel and Frobenius complements of order $p$ (Proposition \ref{pder}).
\end{enumerate}

\ 

We also classify the groups $G$ with $\alpha(G) > 3/4$ (Theorem \ref{trequa}), proving in particular that $3/4$ is the largest non-trivial accumulation point of the set of numbers of the form $\alpha(G)$. An easy consequence of this (Corollary \ref{g-m}) is that if $G$ is not an elementary abelian $2$-group and it has $|G|-n$ cyclic subgroups then $|G| \leq 8n$. This extends and generalizes the results in \cite{tar}, as we show right after the corollary. We also give a formula for $\alpha(G)$ when $G$ is a nilpotent group (Theorem \ref{mnilp}) and we study $\alpha$ of a direct power (Theorem \ref{mpower}) proving that $G^n$ has roughly $|G^n|/\varphi(\exp(G))$ cyclic subgroups.

\section{Basic properties of $\alpha$}

In this section we prove some basic properties of the function $\alpha$.

\subsection{} \label{axb}

If $A$ and $B$ are finite groups of coprime orders then $c(A \times B) = c(A) c(B)$ and hence $\alpha(A \times B) = \alpha(A) \alpha(B)$. The proof of this is straightforward.

\subsection{} \label{2a-1}
Let $I(G)$ denote the number of elements $g \in G$ such that $g^2=1$. Then $$\alpha(G) \leq \frac{1}{2}+\frac{I(G)}{2|G|}, \hspace{2cm} \frac{I(G)}{|G|} \geq 2 \alpha(G)-1.$$ In particular $\alpha(G)=1$ if and only if $G$ is an elementary abelian $2$-group.

\begin{proof}
If $g \in G$ then $g^2=1$ if and only if $\varphi(o(g))=1$, so $$c(G) = \sum_{x \in G} \frac{1}{\varphi(o(x))} \leq I(G)+\frac{1}{2}(|G|-I(G)) = \frac{1}{2}(I(G)+|G|).$$This implies the result.
\end{proof}

\subsection{} \label{inbase} If $N \unlhd G$ then $\alpha(G) \leq \alpha(G/N)$. Moreover $\alpha(G) = \alpha(G/N)$ if and only if $\varphi(o(g)) = \varphi(o(gN))$ for every $g \in G$, where $o(gN)$ denotes the order of the element $gN$ in the group $G/N$.

\begin{proof}
If $a$ divides $b$ then $\varphi(a) \leq \varphi(b)$, therefore
\begin{align*}
c(G/N) & = \sum_{gN \in G/N} \frac{1}{\varphi(o(gN))} = \sum_{g \in G} \frac{1}{|N| \varphi(o(gN))} \geq \sum_{g \in G} \frac{1}{|N| \varphi(o(g))} = \frac{c(G)}{|N|}.
\end{align*}This implies the result.
\end{proof}

\subsection{} \label{el2} If $\alpha(G) = \alpha(G/N)$ then $N$ is an elementary abelian $2$-group.

\begin{proof}
If $n \in N$ then applying \ref{inbase} we have $\varphi(o(n)) = \varphi(o(nN)) = \varphi(o(N)) = \varphi(1)=1$ so $n^2=1$.
\end{proof}

\subsection{} \label{dirf} If $G$ is any finite group then $\alpha(G) = \alpha(G \times {C_2}^n)$ for all $n \geq 0$.

\begin{proof}
Choosing $N = \{1\} \times {C_2}^n$ gives $\varphi(o(x)) = \varphi(o(xN))$ for all $x \in G \times {C_2}^n$. The result follows from \ref{inbase}.
\end{proof}

\subsection{} \label{ln} If $\alpha(G) = \alpha(G/N)$ and $L \unlhd G$, $L \subseteq N$ then $\alpha(G) = \alpha(G/L)$.

\begin{proof}
Since $G/N$ is a quotient of $G/L$ we have $\alpha(G/N) = \alpha(G) \leq \alpha(G/L) \leq \alpha(G/N)$ by \ref{inbase} and the result follows.
\end{proof}

\subsection{} \label{kn} If $\alpha(G) = \alpha(G/N)$ and $K \leq G$ then $\alpha(K) = \alpha(K/K \cap N)$.

\begin{proof}
Let $R := K \cap N$. By \ref{inbase}, is enough to show that if $x \in K$ then $o(xR)=o(xN)$ (because then $\varphi(o(xR))=\varphi(o(xN))=\varphi(o(x))$). Let $a=o(xR)$ and $b=o(xN)$. Since $R \subseteq N$ we have $x^a \in N$ so $b \leq a$. On the other hand $x^b \in K \cap N = R$ so $a \leq b$. Therefore $a=b$.
\end{proof}

\subsection{} \label{nz} Suppose $\alpha(G) = \alpha(G/N)$. If $a \in G$ has order $2$ modulo $N$ then $a$ centralizes $N$, in particular if $G/N$ can be generated by elements of order $2$ then $N \subseteq Z(G)$.

\begin{proof}
We have $\varphi(o(a)) = \varphi(o(aN)) = \varphi(2) = 1$ by \ref{inbase}, and $a$ has order $2$ modulo $N$, so $o(a)=2$. If $n \in N$ then $\varphi(o(an)) = \varphi(o(anN)) = \varphi(o(aN)) = 1$ so $(an)^2=1$. This together with $a^2=n^2=1$ (by \ref{el2}) implies $an=na$. Recalling that $N$ is abelian (by \ref{el2}) we deduce that if $G/N$ can be generated by elements of order $2$ then $N \subseteq Z(G)$.
\end{proof}

\section{A characterization}

Observe that $C_3$ is a quotient of $A_4$ and $\alpha(C_3) = \alpha(A_4) = 2/3$, so it is not always the case that $\alpha(G) = \alpha(G/N)$ implies $G \cong N \times G/N$. We can characterize the groups such that $\alpha(G) = \alpha(G/N)$ when $G/N$ is a symmetric group, for the following two reasons: the symmetric groups can be generated by elements of order $2$ and their double covers are known.

\begin{teor} \label{sym}
Let $G$ be a group and $N$ a normal subgroup of $G$ such that $G/N$ is isomorphic to a symmetric group. If $\alpha(G) = \alpha(G/N)$ then $N$ is an elementary abelian $2$-group and it admits a normal complement in $G$, so that $G \cong N \times G/N$.
\end{teor}

\begin{proof}
We prove the result by induction on the order of $G$. By \ref{el2}, $N$ is an elementary abelian $2$-group. Since $G/N \cong S_m$ can be generated by elements of order $2$, \ref{nz} implies that $N$ is central in $G$. If $m=2$ then the result follows from \ref{2a-1}, so suppose $m \geq 3$. Let $R \cong {C_2}^l$ be a minimal normal subgroup of $G$ contained in $N$. By \ref{ln} we have $\alpha(G/R)=\alpha(G)=\alpha(S_m)$ so by induction, since $G/N$ is a quotient of $G/R$, we have $G/R = {C_2}^l \times S_m$ for some $l \geq 0$. Let $K \unlhd G$ be the (normal) subgroup of $G$ such that $K/R = \{1\} \times S_m$. Observe that $K \cap N$ contains $R$, so $K \cap N/R$ is a normal $2$-subgroup of $K/R \cong S_m$. If $m \neq 4$ this implies that $K \cap N = R$ because $S_m$ in this case does not admit non-trivial normal $2$-subgroups (being $m \geq 3$). If $m=4$ and $K \cap N \neq R$ then $K \cap N/R$ is the Klein group, and $K/K \cap N \cong S_3$. However in this case \ref{inbase} and \ref{kn} imply that
\begin{align*}
17/24 & = \alpha(S_4) = \alpha(K/R) \geq \alpha(K) = \alpha(K/K \cap N) = \alpha(S_3) = 5/6,
\end{align*}
a contradiction. We deduce that $K \cap N = R$.

\ 

If $N \neq R$ then $|K| < |G|$ and $\alpha(K) = \alpha(K/R) = \alpha(S_m)$ by \ref{kn} (being $K \cap N = R$). By induction we deduce that $K \cong R \times S_m$. Set $M := \{1\} \times S_m \leq K$. Since $$S_m \cong G/N \geq KN/N \cong K/K \cap N = K/R \cong S_m$$we obtain $G=KN=MRN=MN$ so being $N$ central in $G$ and $N \cap M = N \cap K \cap M = R \cap M = \{1\}$ we deduce $G = N \times M \cong N \times G/N$. Assume now $N=R$, so $N$ is a minimal normal subgroup of $G$. Since $N$ is central, $|N|=2$ and actually $N = \langle z \rangle = Z(G)$ is the center of $G$ (being $G/N \cong S_m$ with $m \geq 3$). Suppose by contradiction that $G$ is not a direct product $C_2 \times S_m$. We claim that $N$ is contained in the derived subgroup of $G$. Indeed $G'$ is contained in the subgroup $T$ of $G$ such that $T \supseteq N$, $T/N \cong A_m$ (being $|G/T|=2$), so if $G'$ does not contain $N$ then $G'N/N$ is a nontrivial normal subgroup of $G/N \cong S_m$ containing the derived subgroup of $S_m$ (that is, $A_m$) hence $G'N = T$ therefore letting $\varepsilon \in G$ represent a fixed element of order $2$ of $G/N \cong S_m$ not belonging to $A_m$, $\varphi(o(\varepsilon)) = \varphi(2) = 1$ (by \ref{inbase}) hence $o(\varepsilon)=2$ implying that $G' \langle \varepsilon \rangle \cap N = \{1\}$ (otherwise $G' \langle \varepsilon \rangle \supseteq N$ implying that $G' \langle \varepsilon \rangle = G$ so $|G:G'|=2$ hence $G \cong N \times G'$, a contradiction) therefore $G' \langle \varepsilon \rangle \cong G' \langle \varepsilon \rangle N/N = G/N \cong S_m$; being $N$ the center of $G$ we deduce $G \cong C_2 \times S_m$, a contradiction. This implies that $N \subseteq G'$ so $G$ is a double cover of $S_m$ (that is, a stem extension of $S_m$ where the base normal subgroup has order $2$), and looking at the known presentations of the double covers of the symmetric group (classified by Schur, see for example \cite{schur}) we see that $z$ is a square in $G$, that is there exists $x \in G$ with $x^2=z$, so that $x$ has order $4$ and $xN$ has order $2$, contradicting $\varphi(o(x)) = \varphi(o(xN))$ (which is true by \ref{inbase}).
\end{proof}

\section{Nilpotent groups}

Let $G$ be a finite group. For $\ell$ a divisor of $|G|$ let $B_G(\ell)$ be the number of elements $g \in G$ with the property that $g^{\ell}=1$ and let $r_G(\ell)$ be the number of elements of $G$ of order $\ell$. It is worth mentioning the famous result by Frobenius that if $G$ is any group and $\ell$ divides $|G|$ then $\ell$ divides $B_G(\ell)$. The idea of the following result, which is a reformulation of formula (\ref{moe}) in the nilpotent case, is to give a formula for $c(G)$ when $G$ is a nilpotent group in terms of the numbers $B_G(d)$, that in general are reasonably easy to deal with (consider for example the case in which $G$ is abelian).

\begin{teor} \label{mnilp}
If $G$ is a nilpotent group of order $n$ then $c(G) = \sum_{d|n} B_G(d)/d$.
\end{teor}

\begin{proof}
Assume first that $G$ is a $p$-group, $|G|=p^n$. Since $r_G(p^j)=B_G(p^j)-B_G(p^{j-1})$ whenever $j \geq 1$ we see that
\begin{align*}
c(G) & = \sum_{x \in G} \frac{1}{\varphi(o(x))} = \sum_{j=0}^n \frac{r_G(p^j)}{\varphi(p^j)} = 1+\sum_{j=1}^{n} \frac{B_G(p^j)-B_G(p^{j-1})}{\varphi(p^j)} \\ & = 1-\frac{B_G(1)}{\varphi(p)}+\sum_{j=1}^{n-1} \left( \frac{1}{\varphi(p^j)}-\frac{1}{\varphi(p^{j+1})} \right) B_G(p^j)+\frac{B_G(p^n)}{\varphi(p^n)} = \sum_{j=0}^n \frac{B_G(p^j)}{p^j}.
\end{align*}
Now consider the general case, and write the order of $G$ as $n=|G|={p_1}^{n_1} \cdots {p_t}^{n_t}$ with the $p_i$'s pairwise distinct primes, $G$ is a direct product $\prod_{i=1}^t G_{p_i}$ where $G_{p_i}$ is the unique Sylow $p_i$-subgroup of $G$. Using \ref{axb} we obtain $$c(G) = \prod_{i=1}^t c(G_{p_i}) = \prod_{i=1}^t \left( \sum_{j=0}^n \frac{B_G(p_i^j)}{p_i^j} \right) = \sum_{d|n} B_G(d)/d.$$The last equality follows from the fact that since $G$ is nilpotent $B_G(ab)$ equals $B_G(a)B_G(b)$ if $a$ and $b$ are coprime divisors of $n$.
\end{proof}

\section{An asymptotic result}

We want to study $\alpha(G^n)$ where $G^n = G \times G \times \cdots \times G$ ($n$ times) in terms of the functions $B_G$ and $r_G$ defined in the previous section. Recall that the exponent of a group $G$, denoted $\exp(G)$, is the least common multiple of the orders of the elements of $G$. It is clear that $\exp(G^n)=\exp(G)$. The following result shows that $G^n$ has roughly $|G^n|/\varphi(\exp(G))$ cyclic subgroups.

\begin{teor} \label{mpower}
Let $G$ be a finite group. Then $\lim_{n \to \infty} \alpha(G^n) = 1/\varphi(\exp(G))$.
\end{teor}

\begin{proof}
Observe that $r_{G^n}(\ell) \neq 0$ only if $\ell$ divides $|G|$, therefore $$\alpha(G^n) = \frac{1}{|G^n|} \sum_{x \in G^n} \frac{1}{\varphi(o(x))} = \sum_{\ell | |G|^n} \frac{r_{G^n}(\ell)}{\varphi(\ell)|G|^n} = \sum_{\ell | |G|} \frac{r_{G^n}(\ell)}{\varphi(\ell)|G|^n}$$so what we need to compute is the limit $L_{\ell}$ of $r_{G^n}(\ell)/|G|^n$ when $n \to \infty$, for $\ell$ a divisor of $|G|$. Clearly $r_{G^n}(\ell) \leq B_{G^n}(\ell) = B_G(\ell)^n$ hence $r_{G^n}(\ell)/|G|^n \leq (B_G(\ell)/|G|)^n$ so if $B_G(\ell) < |G|$ then $L_{\ell}=0$. Now assume $B_G(\ell)=|G|$, in other words $\exp(G)$ divides $\ell$. If $\exp(G) < \ell$ then $r_{G^n}(\ell)=0$ so $L_{\ell}=0$. Now assume $\exp(G) = \ell$. Let $p$ vary in the set of prime divisors of $|G|$, and for every such $p$ define $a_p := B_G(\exp(G)/p)$. Clearly $G^n$ has at least $|G|^n-\sum_p {a_p}^n$ elements of order $\exp(G)$. Observe that $a_p < |G|$ by definition of $\exp(G)$, so that $a_p/|G| < 1$, hence $(a_p/|G|)^n$ tends to $0$ as $n \to \infty$, implying $L_{\exp(G)}=1$. The result follows.
\end{proof}

\section{A connection with the probability of commutation} \label{cpsection}

The probability that two elements in a group $G$ commute is denoted by $\cp(G)$ (``commuting probability'' of $G$) and is defined by $|S|/|G \times G|$ where $S$ is the set of pairs $(x,y) \in G \times G$ such that $xy=yx$. It is easy to show that $\cp(G)=k(G)/|G|$ where $k(G)$ is the number of conjugacy classes of $G$. This invariant was studied by many authors, but we refer mostly to \cite{gurob}.

\ 

Let $I(G)$ be the size of the set $\{x \in G\ :\ x^2=1\}$. The following lemma is easily deducible from Theorem 2J of \cite{bf}. It can also be proved character-theoretically using the Frobenius-Schur indicator.

\begin{lemma} \label{fs}
$I(G)^2 \leq k(G)|G|$, in other words $\cp(G) \geq (I(G)/|G|)^2$.
\end{lemma}

This together with \ref{2a-1} implies the following inequality.

\begin{lemma} \label{alphacp}
If $\alpha(G) \geq 1/2$ then $\cp(G) \geq (2\alpha(G)-1)^2$.
\end{lemma}

Let us include some other results from \cite{gurob} that we will need in the following section.

\subsection{} \label{solgurob} If $G$ is a non-solvable group and $\sol(G)$ is the maximal normal solvable subgroup of $G$ then $\cp(G) \leq |G:\sol(G)|^{-1/2}$. This follows from \cite[Theorem 9]{gurob} (which depends on the classification of the finite simple groups), and together with Lemma \ref{alphacp} implies that if $\alpha(G) > 1/2$ then $|G:\sol(G)| \leq (2\alpha(G)-1)^{-4}$.

\subsection{} \label{fitgurob} If $G$ is a solvable group and $F(G)$ is the Fitting subgroup of $G$ then $\cp(G) \leq |G:F(G)|^{-1/2}$. This follows from \cite[Theorem 4]{gurob}, and together with Lemma \ref{alphacp} implies that if $\alpha(G) > 1/2$ then $|G:F(G)| \leq (2\alpha(G)-1)^{-4}$.

\section{Non-solvable groups}

\begin{teor} \label{sol}
Let $G$ be a finite non-solvable group. Then $\alpha(G) \leq \alpha(S_5)$ with equality if and only if $G \cong S_5 \times {C_2}^n$ for some integer $n \geq 0$.
\end{teor}

\begin{proof}
Let $\alpha := \alpha(S_5) = 67/120$. We will show that if $G$ is any finite non-solvable group such that $\alpha(G) \geq \alpha$ then $G \cong S_5 \times {C_2}^n$ for some $n \geq 0$. Assume $\alpha(G) \geq \alpha$, in particular $\alpha(G) > 1/2$. By \ref{solgurob} we deduce $|G/\sol(G)| \leq (2 \alpha-1)^{-4} = (60/7)^4 < 5398$ thus $|G/\sol(G)| \leq 5397$. Observe that $G/\sol(G)$ is non-trivial (being $G$ non-solvable), it does not have non-trivial solvable normal subgroups and $\alpha(G/\sol(G)) \geq \alpha(G) \geq \alpha$. If we can show that $G/\sol(G) \cong S_5$ it will follow that $67/120 = \alpha \leq \alpha(G) \leq \alpha(G/\sol(G)) = \alpha(S_5) = 67/120$ therefore $\alpha(G) = \alpha(G/\sol(G))$ and the result follows from Theorem \ref{sym}.

\begin{table}
\begin{tabular}{|l|l|l|l|l|}
\hline
$G$                                                                      & $\Aut(G)$                & $|G|$ & Float($\alpha(G)$) & $\alpha(G)$ \\ \hline
$\Alt(5)$ & $\Sym(5)$                & $60$         & $0.533333$           & $8/15$        \\ \hline
$\Sym(5)$                                                                 & $\Sym(5)$                & $120$        & $0.558333$           & $67/120$      \\ \hline
$\Alt(6)$                                    & $\mbox{P}\Gamma \mbox{L}(2,9)$              & $360$        & $0.463889$           & $167/360$     \\ \hline
$\PGL(2,9)$                                                               & $\mbox{P}\Gamma \mbox{L}(2,9)$              & $720$        & $0.394444$           & $71/180$      \\ \hline
$\Sym(6)$                                            & $\mbox{P}\Gamma \mbox{L}(2,9)$              & $720$        & $0.502778$           & $181/360$     \\ \hline
$\mbox{M}_{10}$                                                                 & $\mbox{P}\Gamma \mbox{L}(2,9)$              & $720$        & $0.419444$           & $151/360$     \\ \hline
$\mbox{P}\Gamma \mbox{L}(2,9)$                                                               & $\mbox{P}\Gamma \mbox{L}(2,9)$              & $1440$       & $0.426389$           & $307/720$     \\ \hline
$\Alt(7)$                                                                 & $\Sym(7)$                & $2520$       & $0.375794$           & $947/2520$    \\ \hline
$\Sym(7)$                                                                 & $\Sym(7)$                & $5040$       & $0.404563$           & $2039/5040$   \\ \hline
$\PSL(3,2)$                                                    & $\PGL(2,7)$              & $168$        & $0.470238$           & $79/168$      \\ \hline
$\PGL(2,7)$                                                    & $\PGL(2,7)$              & $336$        & $0.464286$           & $13/28$       \\ \hline
$\PSL(2,8)$                                                    & $\mbox{P}\Gamma \mbox{L}(2,8)$              & $504$        & $0.309524$           & $13/42$       \\ \hline
$\mbox{P}\Gamma \mbox{L}(2,8)$                                                    & $\mbox{P}\Gamma \mbox{L}(2,8)$              & $1512$       & $0.362434$           & $137/378$     \\ \hline
$\PSL(2,11)$                                                  & $\PGL(2,11)$             & $660$        & $0.369697$           & $61/165$      \\ \hline
$\PGL(2,11)$                                                  & $\PGL(2,11)$             & $1320$       & $0.368182$           & $81/220$      \\ \hline
$\PSL(2,13)$                                                  & $\PGL(2,13)$             & $1092$       & $0.335165$           & $61/182$      \\ \hline
$\PGL(2,13)$                                                  & $\PGL(2,13)$             & $2184$       & $0.322344$           & $88/273$      \\ \hline
$\PSL(2,17)$                                                  & $\PGL(2,17)$             & $2448$       & $0.306373$           & $125/408$     \\ \hline
$\PGL(2,17)$                                                  & $\PGL(2,17)$             & $4896$       & $0.26777$            & $437/1632$    \\ \hline
$\PSL(2,19)$                                                  & $\PGL(2,19)$             & $3420$       & $0.267251$           & $457/1710$    \\ \hline
$\PSL(2,16)$                                                  & $\mbox{P}\Gamma \mbox{L}(2,16)$ & $4080$       & $0.192157$           & $49/255$      \\ \hline
\end{tabular}\caption{Almost-simple groups of order at most $5397$}
\end{table}

We are left to show that if $G$ is a group without non-trivial solvable normal subgroups and such that $|G| \leq 5397$ and $\alpha(G) \geq \alpha(S_5) = 67/120$, then $G \cong S_5$. Let $N$ be a minimal normal subgroup of $G$, then $N=S^t$ with $S$ a non-abelian simple group. If $t \geq 2$ then being $|S| \geq 60$ and $|G| \leq 5397$ we deduce $G=N=A_5 \times A_5$, contradicting the minimality of $N$. So $t=1$. We claim that there is no other minimal normal subgroup of $G$. Indeed if $M$ is a minimal normal subgroup of $G$ distinct from $N$ then $M$ is non-solvable (by assumption) so $|G:MN| = |G|/|MN| \leq 5397/60^2 < 2$ (the smallest order of a non-solvable group is $60$) so $G=MN$ and actually $G = M \times N = A_5 \times A_5$ ($M$ is a direct power of a non-abelian simple group, the smallest orders of non-abelian simple groups are $60$, $168$ and $60 \cdot 60^2,60 \cdot 168$ are both larger than $5397$) which is a contradiction because $\alpha(A_5 \times A_5) = 77/225 < \alpha(S_5)$. We deduce that $N$ is the unique minimal normal subgroup of $G$. Since $N$ is non-solvable, it is non-abelian, so it is not contained in $C_G(N)$, hence $C_G(N)$ must be trivial (otherwise it would contain a minimal normal subgroup of $G$ distinct from $N$) therefore $G$ is almost-simple. Using \cite{atlasas} and \cite{gap} we computed the list of almost-simple groups $G$ of size at most $5397$ and for each of them we determined $\alpha(G)$. The results are summarized in the above table. We deduce that the only almost-simple group $G$ with $|G| \leq 5397$ and $\alpha(G) \geq 67/120$ is $G = S_5$.
\end{proof}

This together with Lemma \ref{alphacp} and \ref{fitgurob} implies the following.

\begin{cor}
If $\alpha(G) > \alpha(S_5)$ then $G$ is solvable and the Fitting subgroup of $G$ has index at most $5397$.
\end{cor}

\section{Groups with many cyclic subgroups}

In this section we will study groups with $\alpha(G)$ ``large'', specifically, we will classify all the finite groups $G$ such that $\alpha(G) > 3/4$. This is a natural choice because $3/4$ turns out to be the largest non-trivial accumulation point of the set of numbers of the form $\alpha(G)$. To do such classification the idea is to observe that if $\alpha(G) > 3/4$ then $I(G)/|G| > 1/2$ and use Wall's classification \cite[Section 7]{wall}.

\begin{teor} \label{trequa}
Let $X$ be a group with $\alpha(X) > 3/4$. Then $X$ is a direct product of an elementary abelian $2$-group with a group $G$, $\alpha(X) = \alpha(G)$, $G$ does not have $C_2$ as a direct factor, and either $G$ is trivial (in which case $\alpha(X)=1$) or one of the following occurs.
\begin{enumerate}
\item Case I. $G \cong A \rtimes \langle \varepsilon \rangle$, where $\langle \varepsilon \rangle = C_2$ acts on $A$ by inversion and there exists an integer $n \geq 1$ such that one of the following occurs. $$A={C_3}^n,\ \alpha(G)=\frac{3 \cdot 3^n+1}{4 \cdot 3^n} \hspace{.5cm} \mbox{or} \hspace{.5cm} A={C_4}^n,\ \alpha(G)=\frac{3 \cdot 2^n+1}{4 \cdot 2^{n}}.$$
\item Case II. $G \cong D_8 \times D_8$ and $\alpha(G) = 25/32$.
\item Case III. $G$ is a quotient ${D_8}^r/N$ where $N = \{(a_1,\ldots,a_r) \in {Z(D_8)}^r\ :\ a_1 \cdots a_r = 1\}$ and $$\alpha(G) = \frac{3 \cdot 2^r+1}{4 \cdot 2^{r}}.$$
\item Case IV. $G$ is a semidirect product $V \rtimes \langle c \rangle$ where $V={\mathbb{F}_2}^{2r}$ has a basis $\{x_1,y_1,\ldots,x_r,y_r\}$, $c$ has order $2$, it acts trivially on each $y_i$, $x_i^c = cx_ic = [c,x_i]x_i = x_iy_i$ for $i=1,\ldots,r$, and $$\alpha(G) = \frac{3 \cdot 2^r+1}{4 \cdot 2^{r}}.$$
\end{enumerate}
\end{teor}

\begin{proof}
We know by \ref{dirf} that $\alpha(X) = \alpha(G)$. Also, we may assume that $G$ is non-trivial. Since $\alpha(G) > 3/4$, by \ref{2a-1} we have $I(G)/|G| \geq 2 \alpha(G)-1 > 1/2$ so $G$ appears in Wall's classification \cite[Section 7]{wall}. Case II is immediate, we will treat cases I, III and IV.

\ 

\textsc{Case I of Wall's classification}. $G$ is a semidirect product $A \rtimes \langle \varepsilon \rangle$ with $A$ an abelian group, $\langle \varepsilon \rangle \cong C_2$ and every element of $G-A$ has order $2$. Observe that $A$ does not admit $C_2$ as a direct factor. Indeed if $a \in A$ then since $a \varepsilon \not \in A$, $a \varepsilon$ has order $2$ so $a^{\varepsilon} = a^{-1}$, hence $\varepsilon$ acts on $A$ as inversion and a direct factor of order $2$ in $A$ would yield a direct factor of order $2$ in $G$. It follows that $c(G) = c(A)+|G|/2$, so that $3/4 < \alpha(G) = \alpha(A)/2+1/2$ implying $\alpha(A) > 1/2$. If the prime $p$ divides the order of $A$ then $C_p$ is a quotient of $A$ so $1/2 < \alpha(A) \leq \alpha(C_p) = 2/p$ whence $p \leq 3$, that is, $p$ is either $2$ or $3$. Write $A = P_2 \times P_3$ where $P_2$ is an abelian $2$-group and $P_3$ is an abelian $3$-group. Observe that $C_9$ is not a quotient of $A$ because otherwise $1/2 < \alpha(A) \leq \alpha(C_9) = 1/3$, a contradiction. Therefore if $P_2$ is trivial then $A = {C_3}^n$ for some $n \geq 1$, an easy computation shows $\alpha(A) = \frac{3^n+1}{2 \cdot 3^n}$, and the result follows. Suppose now that $P_2$ is non-trivial. If $P_3$ is non-trivial then since $P_2$ is not elementary abelian (because $A$ does not have $C_2$ as a direct factor) there is a quotient of $A$ isomorphic to $C_{12}$, however $1/2 < \alpha(A) \leq \alpha(C_{12}) = 1/2$ gives a contradiction. So $P_3 = \{1\}$, in other words $A$ is an abelian $2$-group and we may write $A = \prod_{i=1}^n {C_{2^{a_i}}}$. Since $A$ does not have $C_2$ as a direct factor we deduce $a_i \geq 2$ for all $i$, on the other hand if one of the $a_i$'s is at least $3$ then $C_8$ is a quotient of $A$ but $1/2 < \alpha(A) \leq \alpha(C_8) = 1/2$ is a contradiction. So $A \cong {C_4}^n$ hence an easy computation shows $\alpha(A) = \frac{2^n+1}{2 \cdot 2^n}$, and the result follows.

\ 

\textsc{Case III of Wall's classification}. $G$ is a direct product of $D_8$'s with the centers amalgamated. $G=G(r)$ has a presentation $$G(r) = \langle c,x_1,y_1,\ldots,x_r,y_r\ :\ c^2=x_i^2=y_i^2=1,$$ $$\mbox{all pairs of generators commute except } [x_i,y_i]=c \rangle.$$A more practical description of the group in question is $G = {D_8}^r/N$ where $N = \{(z_1,\ldots,z_r) \in Z^r\ :\ z_1 \cdots z_r=1\}$ where $Z=\langle z \rangle$ (cyclic of order $2$) is the center of $D_8$. $N$ is a subgroup of $Z({D_8}^r)=Z^r$ of index $2$, so $|G|=2 \cdot 4^r$. An element $(a_1,\ldots,a_r)N \in G$ squares to $1$ if and only if $(a_1^2,\ldots,a_r^2) \in N$, that is, $a_1^2 \cdots a_r^2 = 1$. Observe that every $a_i^2$ is either $1$ or $z$, so this condition means that there are an even number of indeces $i$ such that $a_i^2=z$. Since $D_8$ contains $6$ elements that square to $1$ and $2$ elements that square to $z$, ${D_8}^r$ contains exactly $$\beta_r = \sum_{k=0}^{[r/2]} \binom{r}{2k} 2^{2k} 6^{r-2k}$$ elements $(a_1,\ldots,a_r)$ such that $a_1^2 \cdots a_r^2 = 1$. Hence $G$ contains exactly $\beta_r/2^{r-1}$ elements that square to $1$. Observe that $$8^r = (2+6)^r = \sum_{h=0}^r \binom{r}{h} 2^h 6^{r-h}, \hspace{1cm} 4^r = (-2+6)^r = \sum_{h=0}^r \binom{r}{h} (-1)^h 2^h 6^{r-h}$$so adding them together gives exactly $2 \beta_r$. This means that $\beta_r = \frac{1}{2} (8^r+4^r)$, so $G$ has exactly $\beta_r/2^{r-1} = 4^r+2^r$ elements that square to $1$ and exactly $|G|-\beta_r/2^{r-1} = 2 \cdot 4^r-4^r-2^r = 4^r-2^r$ elements of order $4$. Therefore $$\alpha(G) = \frac{1}{2 \cdot 4^r} \left(4^r+2^r+\frac{1}{2} (4^r-2^r) \right) = \frac{3 \cdot 2^r+1}{4 \cdot 2^{r}}.$$

\ 

\textsc{Case IV of Wall's classification}. $G=G(r)$ has a presentation $$G(r) = \langle c,x_1,y_1,\ldots,x_r,y_r\ :\ c^2=x_i^2=y_i^2=1,$$ $$\mbox{all pairs of generators commute except } [c,x_i]=y_i \rangle.$$ A more practical description of the group in question is $G \cong V \rtimes \langle c \rangle$ where $V={C_2}^{2r} = \langle x_1,y_1,\ldots,x_r,y_r \rangle$, $c$ has order $2$, it acts trivially on each $y_i$ and $x_i^c = cx_ic = [c,x_i]x_i = x_iy_i$ for $i=1,\ldots,r$. Thinking of $V$ as a vector space over $\mathbb{F}_2$, if $v \in V$ then $vc$ has order $2$ or $4$, and it has order $4$ exactly when $v^c \neq v$. Observe that with respect to the given basis of $V$ the operator $c$ (acting from the right) has a diagonal block matrix form with $J = \left( \begin{array}{cc} 1 & 1 \\ 0 & 1 \end{array} \right)$ on each diagonal block entry. Thus there are precisely $2^r$ vectors $v$ with $v^c=v$, they are of the form $(0,b_1,0,b_2,\ldots,0,b_r)$. Therefore $G$ has $2^{2r}-2^r$ elements of order $4$ and $c(G) = 2^{2r}+2^r+(2^{2r}-2^r)/2$. This implies that $\alpha(G) = \frac{3 \cdot 2^r+1}{4 \cdot 2^{r}}$.
\end{proof}

The following corollary is immediate. It implies that if $G$ is a non-nilpotent group then $\alpha(G) \leq 5/6$ with equality if and only if $G$ is a direct product ${C_2}^n \times S_3$.

\begin{cor} \label{nil}
Let $G$ be a group such that $\alpha(G) \geq 5/6 = \alpha(S_3)$. Then either
\begin{enumerate}
\item $G \cong {C_2}^n \times S_3$ for some $n \geq 0$ and $\alpha(G)=5/6$, or
\item $G \cong {C_2}^n \times D_8$ for some $n \geq 0$ and $\alpha(G)=7/8$, or
\item $G \cong {C_2}^n$ for some $n \geq 0$ and $\alpha(G)=1$.
\end{enumerate}
\end{cor}

We can deduce a bound of $|G|$ in terms of $|G|-c(G)$. The inequality $\alpha(G) \leq 7/8$ can be written as $|G| \leq 8(|G|-c(G))$, so we obtain the following.

\begin{cor} \label{g-m}
If $G$ is any finite group which is not an elementary abelian $2$-group then $|G| \leq 8(|G|-c(G))$ with equality if and only if $G \cong D_8 \times {C_2}^n$ for some non-negative integer $n$.
\end{cor}

Observe that the above results extend and generalize the results in \cite{tar}. As an example of application let us determine the groups $G$ with $|G|-9$ cyclic subgroups. In this case we have $|G|-c(G)=9$ so $|G| \leq 72$ and a GAP check yields that $G$ is one of $C_{11}$, $D_{22}$ and $C_4 \times S_3$.

\section{Special families of groups}

\begin{prop} \label{odd}
Let $p \geq 3$ be a prime number. Let $G$ be a non-trivial group of order divisible only by primes at least $p$. Then $\alpha(G) \leq 2/p$ with equality if and only if $G \cong C_p$.

In particular if $G$ belongs to the family of groups of odd order then $\alpha(G) \leq 2/3$ with equality if and only if $G \cong C_3$.
\end{prop}

\begin{proof}
If $1 \neq x \in G$ and $q$ is a prime divisor of the order of $x$ then $p \leq q$ so $\varphi(o(x)) \geq \varphi(q) = q-1 \geq p-1$, so since $|G| \geq p$ we have
\begin{align*}
\alpha(G) & = \frac{1}{|G|} \sum_{x \in G} \frac{1}{\varphi(o(x))} \leq \frac{1}{|G|} \left( 1+\frac{|G|-1}{p-1} \right) \\ & = \frac{p-2}{|G|(p-1)}+\frac{1}{p-1} \leq \frac{p-2}{p(p-1)}+\frac{1}{p-1} = \frac{2}{p}.
\end{align*}
If equality holds the above inequalities are equalities and using $p \geq 3$ it is easy to deduce that $|G|=p$, that is, $G \cong C_p$.
\end{proof}

\begin{prop} \label{pder}
Let $G$ be a group and $p$ an odd prime, and suppose $G$ has $C_p$ as epimorphic image (in other words $p$ divides $|G/G'|$). Then $\alpha(G) \leq 2/p$ with equality if and only if $G$ is a direct product of an elementary abelian $2$-group with a Frobenius group with $2$-elementary abelian kernel and complements of order $p$.
\end{prop}

Observe that a bound of $2/p$ when $p=2$ would be trivial. This is why we are only considering the odd case.

\begin{proof}
Since $C_p$ is a quotient of $G$ we have $\alpha(G) \leq \alpha(C_p) = 2/p$. Now assume equality holds. Then $\alpha(G) = \alpha(C_p)$ and there exists $N \unlhd G$ with $G/N \cong C_p$, so $N$ is an elementary abelian $2$-group by \ref{el2}, say $N \cong {C_2}^m$. If $G$ has elements of order $2p$ then it has $C_2$ as a direct factor (for example by Maschke's theorem), so now assume $G$ does not have elements of order $2p$. A subgroup of $G$ of order $p$ acts fixed point freely on $N$ so $G$ is a Frobenius group with Frobenius kernel equal to $N$ and Frobenius complement of order $p$. Now assume $G$ is a Frobenius group with $2$-elementary abelian kernel of size $2^m$ and Frobenius complement of order $p$. The element orders of $G$ are $1$, $2$, and $p$, and $G$ has precisely $2^m-1$ elements of order $2$ and $2^m(p-1)$ elements of order $p$. We have $$\alpha(G) = \frac{1}{|G|} \sum_{x \in G} \frac{1}{\varphi(o(x))} = \frac{1}{2^m \cdot p} \left( 2^m+\frac{2^m(p-1)}{p-1} \right) = 2/p.$$This concludes the proof.
\end{proof}

\section{Non-supersolvable groups}

Let $G$ be a solvable group and let $F_i$ be normal subgroups of $G$ defined as follows: $$F_0=\{1\}, \hspace{1cm} F_{i+1}/F_i := F(G/F_i) \hspace{.5cm} \forall i \geq 0,$$ where $F(G/F_i)$ denotes the Fitting subgroup of $G/F_i$. In particular $F_1$ is the Fitting subgroup of $G$. Since $G$ is solvable, there exists a minimal $h$ such that $F_h=G$, such $h$ is called the ``Fitting height'' of $G$. Observe that $$F(F_l/F_{i-1}) = F_i/F_{i-1} \hspace{1cm} \forall l \geq i \geq 1,$$ indeed $F_i/F_{i-1}$ is nilpotent and normal in $F_l/F_{i-1}$ hence $F_i/F_{i-1} \subseteq F(F_l/F_{i-1})$, and $F(F_l/F_{i-1})$ is nilpotent and characteristic in $F_l/F_{i-1}$, which is normal in $G/F_{i-1}$, so $F(F_l/F_{i-1})$ is normal in $G/F_{i-1}$ hence $F(F_l/F_{i-1}) \subseteq F(G/F_{i-1}) = F_i/F_{i-1}$. This implies in particular that $F_l/F_i$ has Fitting height $l-i$.

\ 

The following consequence of the solution of the $k(GV)$ problem is proved in \cite[Lemma 3, proof of (i)]{gurob} (see also \cite{knorr}).

\begin{prop} \label{kgvcor}
Let $G$ be a group and let $F$ be the Fitting subgroup of $G$. If $G/F$ is nilpotent (that is, $G$ has Fitting height $2$) then $k(G) \leq |F|$, so that $$\cp(G) \leq \frac{1}{|G:F|}.$$
\end{prop}

The following result shows that $S_4$ is a ``maximal'' non-supersolvable group in terms of $\alpha(G)$.

\begin{teor} \label{nonsupersol}
Let $G$ be a group. If $G$ is not supersolvable then $\alpha(G) \leq \alpha(S_4)$ with equality if and only if $G \cong {C_2}^n \times S_4$ for some non-negative integer $n$.
\end{teor}

\begin{proof}
We prove that if $\alpha(G) \geq \alpha(S_4)$ and $G$ is not supersolvable then $G$ is isomorphic to ${C_2}^n \times S_4$ for some non-negative integer $n$. We have $\alpha(G) \geq \alpha(S_4) = 17/24 > 67/120$ so $G$ is solvable by Theorem \ref{sol}, so the Fitting subgroup $F$ of $G$ is non-trivial, and since $G$ is not supersolvable $G \neq F$. Since $2/3 < 17/24$, Proposition \ref{odd} implies that $G$ does not have non-trivial quotients of odd order. Also, since $17/24 > 1/2$ we have $\cp(G) \geq (2 \alpha(G)-1)^2 \geq 25/144$ by Lemma \ref{alphacp}.

\ 

In the following discussion we will use Proposition \ref{kgvcor}, the inequality $\cp(G) \leq \cp(N) \cdot \cp(G/N)$ for $N \unlhd G$ (see \cite[Lemma 1]{gurob}) and the obvious fact that the commuting probability is always at most $1$. Let $F_i$ be the subgroups defined above and let $h$ be the Fitting height of $G$. We distinguish three cases.

\ 

\begin{enumerate}
\item $h = 2$. In this case $G = F_2 > F_1 > \{1\}$. We have that $G/F_1$ is nilpotent so $25/144 \leq \cp(G) \leq 1/|G:F_1|$ so $|G:F_1| \leq 5$. However $|G:F_1| \not \in \{3,5\}$ because $G$ does not have non-trivial quotients of odd order, so $G/F_1$ is one of $C_2$, $C_4$ and $C_2 \times C_2$.

\ 

\item $h = 3$. In this case $G = F_3 > F_2 > F_1 > \{1\}$. We have $25/144 \leq \cp(G) \leq \cp(F_2) \leq 1/|F_2:F_1|$ and $25/144 \leq \cp(G) \leq \cp(G/F_1) \leq 1/|G:F_2|$ so $|F_2:F_1| \leq 5$ and $|G:F_2| \leq 5$. Also $G/F_1$ is not a group of prime power order (because it is not nilpotent) and $|G:F_2|$ is not $3$ or $5$ because $G$ does not have quotients of odd order. Therefore $G/F_1$ is a group of order $6$, $10$, $12$ or $20$, its Fitting subgroup has order at most $5$ and $\alpha(G/F_1) \geq \alpha(G) \geq 17/24$. We deduce $G/F_1 \cong S_3$ by \cite{gap}.

\ 

\item $h \geq 4$. In this case $G \geq F_4 > F_3 > F_2 > F_1 > \{1\}$. We have $$\frac{25}{144} \leq \cp(G) \leq \cp(F_2) \cdot \cp(F_4/F_2) \leq \frac{1}{|F_2:F_1|} \cdot \frac{1}{|F_4:F_3|}$$ so $|F_2:F_1| \cdot |F_4:F_3| \leq 5$ implying that $|F_2:F_1| = |F_4:F_3| = 2$. But then $F(F_3/F_1) = F_2/F_1 \subseteq Z(F_3/F_1)$ implying $F_3/F_1=F_2/F_1$ (because the Fitting subgroup contains its own centralizer), a contradiction.
\end{enumerate}

\ 

We deduce that $G/F$ is one of $C_2$, $C_4$, $C_2 \times C_2$ and $S_3$.

\ 

Since $G$ is not supersolvable there exists a maximal subgroup $M$ of $G$ whose index $|G:M|$ is not a prime number (see \cite[9.4.4]{rob}). Let $M_G$ the normal core of $M$ in $G$, that is, the intersection of the conjugates of $M$ in $G$. Let $X := G/M_G$, $K := M/M_G$, so that $|G:M|=|X:K|$. Then $\alpha(S_4) \leq \alpha(G) \leq \alpha(X)$. This implies that if $X \cong S_4$ then the result follows from Theorem \ref{sym}, so all we have to prove is that $X \cong S_4$. The subgroup $M/M_G$ of $X$ is maximal and it has trivial normal core, so $X$ is a primitive solvable group. We will make use of the known structural properties of primitive solvable groups, see for example \cite[Section 15 of Chapter A]{dt}. $X$ is a semidirect product $X = V \rtimes K$ with $V={C_p}^n$, $p$ a prime, $V$ is the unique minimal normal subgroup of $X$ and it equals the Fitting subgroup of $X$. Since $|V|=|X:K|$ is not a prime, $n \geq 2$, so $|X| > 6 \geq |G/F|$ hence $F \not \subseteq M_G$ so $FM_G/M_G$ is a non-trivial nilpotent normal subgroup of $X$ so it equals $V$, hence $$K \cong X/V = (G/M_G)/(FM_G/M_G) \cong G/FM_G$$ is a quotient of $G/F$ so $K$ is one of $C_2$, $C_4$, $C_2 \times C_2$ and $S_3$.

\ 

In what follows we will use the known representation theory of small groups over the field with $p$ elements. We will think of $V$ as a vector space of dimension $n$ over the field $\mathbb{F}_p$, irreducible when seen as a $\mathbb{F}_p[K]$-module.

\ 

Suppose $p$ divides $|K|$. Observe that being not supersolvable, $X$ is not a $2$-group, hence $K$ cannot be $C_2$, $C_4$ nor $C_2 \times C_2$, so $K \cong S_3$. The structure of the group algebras $\mathbb{F}_2[S_3]$ and $\mathbb{F}_3[S_3]$ implies that $n \geq 2$ forces $n=p=2$ and hence $X \cong S_4$.

\ 

Suppose now $p$ does not divide $|K|$, so that $|K|$ and $|V|$ are coprime.

\ 

Suppose $\mathbb{F}_p$ is a splitting field for $K$. Since $n \geq 2$, the only possibility is $K \cong S_3$, $n=2$, and the action of $K$ on $V$ defining the group structure of $X = V \rtimes K$ is the following: $K \cong S_3$ permutes the coordinates of the vectors in the fully deleted module $$V = \{(a,b,c) \in {\mathbb{F}_p}^3\ :\ a+b+c=0\}.$$ The elements of order $2$ in $X$ are of the form $(a,b,c)k$ with $k \in K$, $o(k)=2$ and the fixed coordinate is zero, so there are $3p$ of them. The elements of order $3$ in $X$ are of the form $vk$ with $v \in V$ arbitrary and  $k \in K$, $o(k)=3$, so there are $2p^2$ of them. The elements of order $p$ in $X$ are the non-trivial elements of $V$ (being $p$ coprime to $|K|=6$) so there are $p^2-1$ of them. The elements of order $2p$ in $X$ are of the form $vk$ with $v \in V$, $k \in K$, $o(k)=2$ and $o(vk) \neq 2$ so there are $3p^2-3p$ of them. The exponent of $X$ is $6p$ and $X$ has no elements of order $6$, $3p$ or $6p$. This implies that
\begin{align*}
c(X) & = \sum_{x \in X} \frac{1}{\varphi(o(x))} = 1+3p+\frac{2p^2}{2}+\frac{p^2-1}{p-1}+\frac{3p^2-3p}{p-1} = p^2+7p+2.
\end{align*}
Using $p \geq 5$ and $|X|=6p^2$ we deduce $\alpha(X) < 17/24$, a contradiction.

\ 

Suppose $\mathbb{F}_p$ is not a splitting field for $K$. Since $n \geq 2$, the only possibility is $K \cong C_4$, $n=2$, and the polynomial $t^2+1$ does not split modulo $p$, that is, $p \equiv 3 \mod 4$. Let $x$ be a generator of the cyclic group $K$. We may interpret $x$ as a matrix of order $4$, so that by irreducibility the minimal polynomial of $x$ is $t^2+1$ (the only irreducible factor of degree $2$ of $t^4-1$). This implies that $x^2=-1$. Choosing a nonzero vector $v_1 \in V$ and $v_2:={v_1}^x$, since $v_1,v_2$ are linearly independent (otherwise $v_1$ would be an eigenvector for $x$ contradicting irreducibility), and ${v_2}^x={v_1}^{x^2}=-v_1$, the matrix of $x$ in the base $\{v_1,v_2\}$ is $\left( \begin{array}{cc} 0 & 1 \\ -1 & 0 \end{array} \right)$ acting by right multiplication. $X = V \rtimes K$ has $p^2-1$ elements of order $p$ (the non-trivial elements of $V$, being $p$ coprime to $|K|=4$), $p^2$ elements of order $2$ (the elements of the form $vx^2$ with $v \in V$ arbitrary) and $2p^2$ elements of order $4$ (the elements of the form $vx$ or $vx^3$ with $v \in V$ arbitrary). The exponent of $X$ is $4p$ and there are no elements of order $2p$ or $4p$, therefore $$c(X) = \sum_{x \in X} \frac{1}{\varphi(o(x))} = 1+p^2+\frac{2p^2}{2}+\frac{p^2-1}{p-1} = 2p^2+p+2.$$Using $p \geq 3$ and $|X|=4p^2$ we deduce $\alpha(X) < 17/24$, a contradiction.
\end{proof}

\section{Acknowledgements}

We would like to thank the referee for carefully reading a previous version of this paper and having very useful suggestions.

\end{document}